\documentclass{amsart}
\usepackage{amssymb, latexsym}
\usepackage{lmodern,microtype}

\begin{document}
\bibliographystyle{plain}

\newtheorem*{definition}{Definition}
\newtheorem{theorem}{Theorem}
\newtheorem{lemma}{Lemma}
\newtheorem{corollary}{Corollary}

\title{Fully-projected subsets}
\author{Jason Gibson}
%\thanks{}
\address{Department of Mathematics and Statistics,
Eastern Kentucky University, KY 40475, USA}
\email{jason.gibson@eku.edu}

\begin{abstract}
  Let $k$ and $i_1,\ldots,i_n$ be natural numbers.
  Place $k$ balls into a multidimensional box
  of $i_1\times\cdots \times i_n$ cells,
  no more than one ball to each cell, such
  that the projections to each of the coordinate
  axes have cardinalities $i_1,\ldots,i_n$, respectively.
  We generalize earlier work of Wang, Lee, and
  Tan to find a formula for the alternating sum
  of the number of these fully-projected subsets.
  
\end{abstract}

\maketitle

\section{Introduction}

Let $k$ and $i_1,\ldots,i_n$ be natural numbers.
We consider the task of placing $k$ balls into
a multidimensional box of $i_1\times\cdots \times i_n$
cells, such that each cell contains
at most one ball, and such that each projection
to each coordinate has cardinality $i_1,\ldots, i_n$,
respectively. This generalizes work of Wang, Lee,
and Tan \cite{wang_lee_tan} on the two-dimensional
version of the problem, where the condition
requires that each row and each column contains
at least one ball. Some of their interest in the formula
stemmed from its role in the theory of
falling random subsets in fuzzy statistics.

If we call such subsets fully-projected,
then, generalizing the result of \cite{wang_lee_tan},
we have the following formula
involving the alternating sum of
the numbers $t_k$
of these subsets.

\begin{theorem}\label{main}
  Let $k$ and $i_1,\ldots,i_n$ be natural numbers,
  and, for $j=1,\ldots,n$, let $I_j=\{1,\ldots,i_j\}$.
  If $t_k=t_k(i_1,\ldots,i_n)$ denotes the number of all fully-projected $k$-subsets of
  $\prod_{j=1}^nI_j$, then
  \begin{equation}
\sum_{k=1}^{i_1\cdots i_n} (-1)^{k-1}t_k = (-1)^{i_1+\cdots+i_n}.
  \end{equation}
\end{theorem}
Our proof of Theorem~\ref{main} follows the approach
of Wang, Lee, and Tan.
The combinatorial analysis here, provided in Section~\ref{count} below,
requires a small bit of care in order
to avoid a blurred forest of unions and intersections
over the index sets and elements.

Work of Fulmek \cite{fulmek} generalized the result of Wang, Lee, and Tan
in a different direction, leading to an interpretation of the formula
in the language of dual rook polynomials.
The rook polynomial of a board $\mathfrak{B}$ (an arbitrary
subset of the cells of an $m\times n$ array) is
defined by
\begin{equation}
  P_{\mathfrak{B}}(x)=\sum_{k\ge 0} R_k(\mathfrak{B})x^k,
\end{equation}
where $R_k(\mathfrak{B})$ is the number of ways
to place $k$ non-attacking rooks on the board $\mathfrak{B}$.
The property of non-attacking can be viewed
as the requirement that each row and each column
contains \emph{at most one} rook. Fulmek considered
a sort of dual notion. Letting $\tilde{R}_k(\mathfrak{B})$
denote the number of ways to place $k$ rooks on $\mathfrak{B}$
such that each row and each column contains
\emph{at least one}
rook, Fulmek called the polynomial
\begin{equation}
  \tilde{P}_{\mathfrak{B}}(x)=\sum_{k\ge 0} \tilde{R}_k(\mathfrak{B})x^k
\end{equation}
the dual rook polynomial of $\mathfrak{B}$.
A key result from \cite{fulmek}, generalizing
the Wang, Lee, and Tan formula, gives that
$\tilde{P}_{\mathfrak{B}}(-1)$ is always $-1$, $0$, or $1$
for skew Ferrers boards.

Fulmek's paper also contains some interesting conjectures
related to these matters, including, e.g., the question
of the log-concavity of the dual rook numbers $\tilde{R}_k(\mathfrak{B})$.
The resolution of those conjectures in their original
formulation (or the consideration of appropriate multidimensional
generalizations) and the finer combinatorial and statistical
properties of the numbers $t_k$ present multiple avenues
for further work.

\section{Counting via inclusion-exclusion}\label{count}

To aid in the combinatorial analysis,
we begin with a definition of fully-projected
subset that clarifies the projection property.
The proof of Theorem~\ref{main} appears following
this definition.

\begin{definition}[Fully-projected $k$-subset]
  Let $S_k\subseteq \prod_{j=1}^nI_j$ be
  a $k$-element subset of $\prod_{j=1}^nI_j$.
  Call $S_k$ a fully-projected subset of
  $\prod_{j=1}^nI_j$, denoted by
  $S_k \hookrightarrow \prod_{j=1}^nI_j$, provided that,
  for $j=1,\ldots,n$,
  the set $S_k$ satisfies
  \begin{equation}
    \pi_j(S_k) = I_j=\{1,\ldots,i_j\}.
  \end{equation}
  Here $\pi_j$ denotes projection onto the $j$th coordinate,
  so that $\pi_j(a_1,\ldots,a_n)=a_j$.
  \end{definition}

\begin{proof}[Proof of Theorem~\ref{main}]
  Let $Q$ denote the set of all $k$-subsets of $M=\prod_{j=1}^nI_j$,
  and let $A$ denote the set of all fully-projected $k$-subsets of $M$.
  Further, for $j=1,\ldots,n$ and $r=1,\ldots,i_j$,
  let $B_{j,r}$ denote the set of $k$-element subsets of $M$
  that avoid element $r$ within coordinate $j$.
  Succinctly, we have
  \begin{align}
    \begin{split}
      Q&=\{S_k:S_k\subseteq M\},\\
      A&=\{S_k:S_k\hookrightarrow M\},\\
      B_{j,r}&=\{S_k:S_k\subseteq I_1\times\cdots\times (I_j\backslash
      \{r\})\times\cdots\times I_n\}.
    \end{split}
  \end{align}

  Note that $|Q|=\binom{i_1\ldots i_n}{k}$.
  Also, from the above, we see that
  \begin{equation}
    A=Q\backslash\left(\bigcup_{j=1}^n\bigcup_{r=1}^{i_j}B_{j,r}\right),
  \end{equation}
  and
  \begin{equation}\label{diff}
    t_k = |A| = |Q| - \left|\bigcup_{j=1}^n\bigcup_{r=1}^{i_j}B_{j,r}\right|,
  \end{equation}
  because the fully-projected $k$-element subsets collected in $A$ are exactly
  the $k$-element subsets that, together, miss no element in any
  coordinate.

  Define $\alpha$ by
  \begin{equation}\label{alpha}
    \alpha = \left|\bigcup_{j=1}^n\bigcup_{r=1}^{i_j}B_{j,r}\right|.
  \end{equation}
  Then, by the inclusion-exclusion principle,
  letting the index sets $J_j$ range over
  subsets of $I_j=\{1,\ldots,i_j\}$ and using $\sum^{'}$
  to indicate a sum that excludes the case $m_1=\ldots=m_n=0$,
  we have that
  \begin{align*}
    \alpha
    &=
    \sideset{}{'}
    \sum_{\substack{0\le m_j \le i_j\\\text{for\ }j=1,\ldots,n}}
    (-1)^{m_1+\cdots+m_n-1}
    \sum_{\substack{J_1\subseteq I_1,\ldots, J_n\subseteq I_n\\
    |J_1|=m_1,\ldots,|J_n|=m_n}}
    \left|\bigcap_{j=1}^n \bigcap_{r_j\in J_j}B_{j,r_j}\right|\\
    &=
    \sideset{}{'}
    \sum_{\substack{0\le m_j \le i_j\\\text{for\ }j=1,\ldots,n}}
    (-1)^{m_1+\cdots+m_n-1}
    \sum_{\substack{J_1\subseteq I_1,\ldots, J_n\subseteq I_n\\
    |J_1|=m_1,\ldots,|J_n|=m_n}} \binom{(i_1-m_1)\cdots (i_n-m_n)}{k}
    \\
    &=
    \sideset{}{'}
    \sum_{\substack{0\le m_j \le i_j\\\text{for\ }j=1,\ldots,n}}
    (-1)^{m_1+\cdots+m_n-1}
    \binom{i_1}{m_1}\cdots\binom{i_n}{m_n}
     \binom{(i_1-m_1)\cdots (i_n-m_n)}{k}.
  \end{align*}

  We have then, by \eqref{diff}, \eqref{alpha}, and the above
  expression for $\alpha$, that
  \begin{align}
    t_k &=|Q|-\alpha\\
    &=\binom{i_1\cdots i_n}{k} - \alpha\\
    &=\sum_{\substack{0\le m_j \le i_j\\\text{for\ }j=1,\ldots,n}}
    (-1)^{m_1+\cdots+m_n}
    \binom{i_1}{m_1}\cdots\binom{i_n}{m_n}
     \binom{(i_1-m_1)\cdots (i_n-m_n)}{k}.\label{t}
  \end{align}

  Using \eqref{t}, we obtain that
  \begin{align*}
    \sum_{k=1}^{i_1\cdots i_n} &(-1)^{k-1}t_k 
    \\
    &=\sum_{k=1}^{i_1\cdots i_n} (-1)^{k-1}
    \sum_{\substack{0\le m_j \le i_j\\\text{for\ }j=1,\ldots,n}}
    (-1)^{m_1+\cdots+m_n}
    \binom{i_1}{m_1}\cdots\binom{i_n}{m_n}
    \binom{(i_1-m_1)\cdots (i_n-m_n)}{k}\\
    &=\sum_{\substack{0\le m_j \le i_j\\\text{for\ }j=1,\ldots,n}}
    (-1)^{m_1+\cdots+m_n}
    \binom{i_1}{m_1}\cdots\binom{i_n}{m_n}
    \sum_{k=1}^{i_1\cdots i_n} (-1)^{k-1}
    \binom{(i_1-m_1)\cdots (i_n-m_n)}{k}\\
    &=\sum_{\substack{0\le m_j < i_j\\\text{for\ }j=1,\ldots,n}}
    (-1)^{m_1+\cdots+m_n}
    \binom{i_1}{m_1}\cdots\binom{i_n}{m_n}
    \sum_{k=1}^{i_1\cdots i_n} (-1)^{k-1}
    \binom{(i_1-m_1)\cdots (i_n-m_n)}{k}\\
    &=\sum_{\substack{0\le m_j < i_j\\\text{for\ }j=1,\ldots,n}}
    (-1)^{m_1+\cdots+m_n}
    \binom{i_1}{m_1}\cdots\binom{i_n}{m_n}\cdot 1\\
    &=\prod_{j=1}^n\left(\sum_{0\le m_j<i_j}(-1)^{m_j}\binom{i_j}{m_j}\right)\\
    &=(-1)^{i_1+\cdots+i_n},
  \end{align*}
  which completes the proof of Theorem~\ref{main}.
\end{proof}

%\bibliography{fps}{}

\end{document}